\documentclass[11pt,oneside,leqno]{amsart}
\usepackage{amsxtra}
\usepackage{amsopn}
\usepackage{amsmath,amsthm,amssymb}
\usepackage{color}
\usepackage{amscd}
\usepackage{amsfonts}
\usepackage{latexsym}
\usepackage{verbatim}
\usepackage{pb-diagram}
\usepackage{pifont}
\usepackage{graphicx}
\usepackage{array}

\newcommand{\xmark}{\ding{55}}

\newcommand{\om}{\omega}

\theoremstyle{plain}
\newtheorem{theorem}{Theorem}[section]
\newtheorem*{theorem*}{Theorem}

\newtheorem{rem}[theorem]{Remark}

\newtheorem*{mt*}{Main Theorem}
\newcommand\g{{\mathfrak{g}}}
\newcommand\h{{\mathfrak{h}}}

\newcommand\m{{\mathfrak{m}}}

\newcommand{\D}{\mathcal D}

\begin{document}
\title[On Conformally flat homogeneous Walker four-manifolds]{On Conformally flat homogeneous Walker four-manifolds}
\author{M. Chaichi, A. Zaeim and Y. Keshavarzi}
\date{}

\address{Department of Mathematics\\Payame noor University\\P.O. Box 19395-3697\\Tehran\\Iran.}
\email{Mohammad Chaichi: chaichi@pnu.ac.ir, Amirhesam Zaeim: zaeim@pnu.ac.ir, Yadollah Keshavarzi: y.keshavarzi@pnu.ac.ir}


\begin{abstract}
In this paper we study the invariant Walker structures over the conformally flat four-dimensional homogeneous manifolds according to the Segre types of the Ricci operator.
\end{abstract}

\maketitle

\section{Introduction}
Conformally flat spaces are the subject of many investigations in Riemannian and pseudo-Riemannian geometry. A conformally flat (locally) homogeneous Riemannian manifold is (locally) symmetric \cite{Ta}, and so, as proved in \cite{Ry}, it admits an universal covering either a space form $\mathbb R^n$, $\mathbb S^n(k)$, $\mathbb H^n(-k)$, or one of the Riemannian products $\mathbb R \times \mathbb S^{n-1}(k)$, $\mathbb R\times \mathbb H^{n-1}(-k)$ and $\mathbb S^{p}(k)\times \mathbb H^{n-p}(-k)$.

In the pseudo-Riemannian setting the problem is more complicated and of course interesting. In dimension three, the conformally flat examples were classified independently in \cite{Ho} and \cite{Ca}, where contrary to the Riemannian case they showed the existence of non-symmetric examples. By expanding the results of \cite{Ho}, the same authors solved the classification problem for the Lorentzian manifolds of any dimension with diagonalizable Ricci operator \cite{Ho2}.
For homogeneous spaces, the classification problem has been completely solved for both Lorentzian and neutral signatures in dimension four \cite{Ca1}. A fundamental step for this classification was to determine the forms ({\em Segre types}) of the Ricci operator. Homogeneous spaces are the subject of many interesting research projects in the pseudo-Riemannian framework. Four-dimensional homogeneous Lorentzian and neutral signature manifolds were studied in \cite{Ca2} and \cite{ChZ1} respectively and Lorentzian Lie groups with complete classification of Einstein and Ricci parallel examples were considered in dimension four in \cite{Ca3}.

A pseudo-Riemannian manifold which admits a parallel degenerate distribution is called a {\em Walker} manifold. Walker spaces were introduced by Arthur Geoffrey Walker in 1949. The existence of such structures causes many interesting properties for the manifold with no Riemannian counterpart. Walker also determined a standard local coordinates for these kind of manifolds \cite{Wa1,Wa2}.

 In this paper, which is based on the study of conformally flat spaces in \cite{Ca1}, we have determined invariant Walker structures in the case of four-dimensional conformally flat homogeneous manifolds. Conformally flat homogeneous spaces have been studied classically in pseudo-Riemannian geometry. As it is known, the existence of Walker structures on a manifold can be responsible for the existence of non-symmetric examples. So we analyze the conformally flat homogeneous pseudo-Riemannian Walker four-manifolds.

 %At first Honda in \cite{} classified n$(\geq3)$-dimensional have conformally flat homogeneous pseudo-Riemannian manifold with diagonalizable Ricci operator. Recently the four-dimensional conformally flat homogeneous pseudo-Riemannian manifold was classified completly by Calvaruso in \cite{zaeim} which is dependent on forms (segre types) of the Ricci operator. Since the possible degenerate forms are realized by some homogeneous spaces with non-trivial isotropy, for which we use Komrakov s classification \cite{} to deduce all possible examples.

  The paper is organized as follows. We recall some basic facts about the Ricci operator of a Conformally flat homogeneous four-dimensional manifold in Section 2. In Section 3 according to the forms (Segre types) of the Ricci operator we study the left-invariant Walker structures of cases with non-degenerate Ricci operator. The cases with degenerate Ricci operator and trivial isotropy invariant Walker structures will be studied in Section 4. In Section 5 we consider invariant Walker structures on the cases with degenerate Ricci operator and non-trivial isotropy.

\section{preliminaries}
Let $(M,g)$ be a pseudo-Riemannian manifold of dimension $n\geq3$ and $\nabla$ its Levi-Civita connection. We use the curvature tensor with the sign convention $R(X,Y)=[\nabla_{X},\nabla_{Y}]-\nabla_{[X,Y]}$ for all vector
fields $X$, $Y$ on $M$. The Ricci tensor is given by the identity
\begin{eqnarray}\label{Ric}
\varrho(X,Y)=\sum_{i=1}^4\varepsilon_i g(R(e_i,X)Y,e_i),
\end{eqnarray}
for all $X, Y \in T_pM$, where $\{e_1,e_2,e_3,e_4\}$ is a pseudo-orthonormal basis for the tangent space $T_pM$. We denote the Ricci operator and the scalar curvature by $Q$ and $\tau$ respectively. Let $p$ be a point of $M$ and $\lbrace e_{1},...,e_{n}\rbrace$
an orthonormal basis of the tangent space $T_{p}M$. It is well-known that
for a conformally
flat space the curvature tensor can be completely determined using the Ricci tensor by the identity
\begin{equation}\label{CurvFromRic}
\begin{array}{cccc}
R_{ijkh}=\frac{1}{n-2}(g_{ih}\varrho_{jk}+g_{jk}\varrho_{ih}
-g_{ik}\varrho_{jh}-g_{jh}\varrho_{ik})\\
-\frac{\tau}{(n-1)(n-2)}
(g_{ih}\varrho_{jk}-g_{ik}\varrho_{jh}).
\end{array}
\end{equation}
Moreover, the Equation \eqref{CurvFromRic} characterizes conformally
flat pseudo-Riemannian manifolds of dimension
$n\geq4$, while it is trivially satisfied by any three-dimensional manifold. Conversely,
the condition
\begin{eqnarray}
\begin{array}{cr}
\nabla_{i} \varrho_{jk}-\nabla_{j} \varrho_{ik}=\frac{1}{2(n-2)}(g_{jk}\nabla_{i}\tau-g_{ik}\nabla_{j}\tau),
\end{array}
\end{eqnarray}
which characterizes three-dimensional conformally flat spaces, is trivially satisfied
by any conformally flat Riemannian manifold of dimension greater than three.

Now, let $(M,g)$ be a locally homogeneous pseudo-Riemannian manifold. Then, for each pair of points {$p,p' \in M$}, there exists a local isometry {$\phi$} between neighbourhoods of $p$ and {$p'$}, such that {$\phi(p)=p'$}. In particular, for any choice of an index $k$, {$\phi^*:T_{p'}M\rightarrow T_pM$} satisfies {$\phi^*(\nabla^i R_{p'})=\nabla^i R_p$} for all $i=0,..,k$. Consequently, chosen a pseudo-orthonormal basis $\{e_i\}_p$ for $T_pM$, by means of the isometries between $p$ and any other point {$p'\in M$}, one can build a pseudo-orthonormal frame field $\{e_i\}$ on $M$, with respect to which the components of the curvature tensor and its covariant derivatives up to order $k$  are globally constant on $M$.

In the special case when $(M,g)$ is conformally flat, this is equivalent to determining a pseudo-orthonormal frame field $\{e_i \}$ on $(M,g)$, such that the components of the Ricci tensor $\varrho$ and its covariant derivatives  $\nabla^i\varrho$, for $i=1,\dots,k$, are constant globally on $M$. To note that in particular, with respect to $\{e_i \}$, the components of the Ricci operator $Q$ are constant. Specially the Segre type of the Ricci operator stays constant on the whole space.

Following \cite{Ca1}, we now recall the possible Segre types of the Ricci operator for a conformally flat homogeneous four-dimensional manifold through the following tables.\\
%If $g$ is of signature $(2,2)$,then there exists a pseudo-orthonormal frame field
%$\lbrace e_{1},e_{2},e_{3},e_{4}\rbrace$ with $e_{3}$, $e_{4}$ time-like vector fields, such that the linear operator A takes one of the following forms: \\
%
%In particular when $g$ is Lorentaian there exist a pseudo-orthonormal basis $\lbrace e_{1},e_{2},e_{3},e_{4}\rbrace$ with $e_{3}$ timelike such that A takes one of the following forms:\\
%
%
%So because of above comment, the only possible Segre types of the Ricci operator $Q$ (equiv-
%alently, of $A$) for a conformally
%at homogeneous four-dimensional manifold, are the ones
%listed in the following tables.\\
%
%
%
%
\begin{center}
{\bf Table I: Segre types of $Q$ for an inner product of signature ${\bf (2,2)}$.}\nopagebreak \\[3 pt]
\begin{tabular}{!{\vrule width 2pt}l|c|c|c|c|c!{\vrule width 2pt}}
\noalign{\hrule height 2pt}
Case $\vphantom{\displaystyle A^{{A}^{A}}}$
&Ia&Ib&Ic&IIa&IIb\\[2 pt]
\hline
Non-degenerate type $\vphantom{\displaystyle A^{{A}^{A}}}$ &---&$[1,11\bar 1]$&$[1\bar 11\bar 1]$&---&$[22]$\\[2 pt]
\hline
Degenerate types $\vphantom{\displaystyle A^{{A}^{A}}}$
&$[(11),(11)]$&$[(1,1)1\bar 1]$&$[(1\bar 11\bar 1)]$&$[(1,1)2]$&$[(22)]$\\
&$[(1|(1,1)|1)]$&&&$[1,(12)]$&\\
&$[(11,1)1]$&&&$[(1,12)]$&\\
&$[1(1,11)]$&&&&\\
&$[(11,11)]$ &&&&\\[2 pt]
\hline
Case $\vphantom{\displaystyle A^{{A}^{A}}}$
&IIc&IId&IIIa&IIIb&IV\\
\hline
Non-degenerate type $\vphantom{\displaystyle A^{{A}^{A}}}$ &$[21\bar 1]$&$[2\bar2]$&$[13]$&[1,3]&$[4]$\\[2 pt]
\hline
Degenerate types $\vphantom{\displaystyle A^{{A}^{A}}}$ &---&---&$[(13)]$&[(1,3)]&---\\[2 pt]
\noalign{\hrule height 2pt}
\end{tabular}  %\nopagebreak \\ \nopagebreak
%\vspace{.1cm}
%Table I: Segre types of $Q$ for an inner product of signature $(2,2)$.

%\pagebreak
\medskip
{\bf Table II: Segre types of $Q$ for a Lorentzian inner product.}\nopagebreak  \\[3 pt]
\begin{tabular}{!{\vrule width 2pt}l|c|c|c|c!{\vrule width 2pt}}
\noalign{\hrule height 2pt}
Case $\vphantom{\displaystyle A^{{A}^{A}}}$ &Ia&Ib&II&III\\[2 pt]
\hline
Non-degenerate type $\vphantom{\displaystyle A^{{A}^{A}}}$ &---&$[11,1\bar 1]$&---&[1,3]\\[2 pt]
\hline
Degenerate types $\vphantom{\displaystyle A^{{A}^{A}}}$ &$[(11)(1,1)]$&$[(11),1\bar 1]$&$[(11),2]$&$[(1,3)]$\\
&$[1(11,1)]$&&$[1(1,2)]$&\\
&$[(111),1]$&&$[(11,2)]$&\\
&$[(111,1)]$&&&\\[2 pt]
\noalign{\hrule height 2pt}
\end{tabular}
\vspace{.1cm}
%\nopagebreak \\ \nopagebreak
%Table II: Segre types of $Q$ for a Lorentzian inner product.
\end{center}
By the result of \cite{Ho}, for a conformally flat homogeneous manifold of dimension four with digonalizable Ricci operator, the problem of study Walker structures reduces to the well known space forms.

\begin{theorem}\cite{Ho}\label{CFdiag}
Let $M_q^n$ be an $n(\geq 3)$-dimensional conformally flat homogeneous pseudo-Riemannian manifold with diagonalizable Ricci operator. Then, $M_q^n$ is locally isometric to one of the following:
\begin{itemize}
\item[(i)] A pseudo-Riemannian space form;
\item[(ii)] A product manifold of a $m$-dimensional space form of constant curvature $k \neq 0$ and a $(n-m)$-dimensional pseudo-Riemannian manifold of constant curvature $-k$, where $2\leq m\leq n-2$;
\item[(iii)] A product manifold of a $(n-1)$-dimensional pseudo-Riemannian manifold of index $q - 1$ of constant curvature $k\neq 0$ and an one-dimensional Lorentzian manifold, or a product of a $(n-1)$-dimensional pseudo-Riemannian manifold of index $q$ of constant curvature $k \neq 0$ and an one-dimensional Riemannian
manifold.
\end{itemize}
\end{theorem}
It is obvious from the last theorem that if $(M,g)$ have digonalizable Ricci operator then the Ricci operator is degenerate. So the study of cases with non-degenerate Ricci operator restricts to the not diagonalizable ones.

\section{Cases with non-degenerate Ricci operator}
Let $(M,g)$ be a conformally flat homogeneous four dimensional manifold with non-degenerate Ricci operator.
For any point $p\in M$, we have that $g(0, p)=\{ 0\}$  if and only if $ Q_{p} $ is
non-degenerate. Therefore, $(M,g)$ is locally isometric to a Lie group equipped with a left-invariant pseudo-Riemannian metric and the Ricci operator of conformally flat homogeneous pseudo-Riemannian four-manifolds can only be of Segre type $[1,11\bar{1}]$ if $g$ is neutral, or $[11,1\bar{1}]$ if $g$ is Lorentzian \cite{Ca1}. The Lie group structure of the mentioned types could be realized by the following theorems.

\begin{theorem}\cite{Ca1}\label{ND-NU}
Let $(M,g)$ be a conformally flat homogeneous four-dimensional manifold with the Ricci operator of Segre type $[1,11\bar1]$. Then, $(M,g)$ is locally isometric to {one of the unsolvable Lie groups $SU(2) \times \mathbb R$ or $SL(2,\mathbb R) \times \mathbb R$}, equipped with a left-invariant neutral metric, admitting a pseudo-orthonormal basis  $\{e_1,e_2,e_3,e_4\}$ for their Lie algebra, %$\mathfrak{r} \cdot \mathfrak{su}(2)$,
such that the Lie brackets take one of the following forms:

$$%\begin{equation}\label{Alg[1,1zz]}
\begin{array}{llll}
{\rm i)} \quad & [e_1,e_2]=\varepsilon\alpha e_3,\quad &[e_1,e_3]=-\varepsilon\alpha e_2,\quad & [e_2,e_3]=2\alpha(e_1+\varepsilon e_4),\\[2pt] &
[e_2,e_4]=-\alpha e_3,\quad & [e_3,e_4]=\alpha e_2, &
\end{array}
$$%\end{equation}
$$%[ii)] \begin{equation}
\begin{array}{llll}
{\rm ii)} \quad & [e_1,e_2]=-\varepsilon\alpha e_1,\quad &[e_1,e_3]=\alpha e_1,\quad & [e_1,e_4]=2\alpha(\varepsilon e_2-e_3),\\[2pt] &
[e_2,e_4]=-\varepsilon\alpha e_4,\quad & [e_3,e_4]=\alpha e_4,&
\end{array}
$$%\end{equation}
where $\alpha\neq 0$ is a real constant and $\varepsilon=\pm1$.
\end{theorem}
%\begin{theorem}
%\cite{zaeim} Let $(M,g)$ be a conformally flat homogeneous pseudo-Riemannian four manifold. if the Ricci operator
%$Q$ of $(M,g)$ is not diagonalizable and degenerate, then $Q$ can only be of Segre type $[1,11\bar{1}]$ if $g$ is neutral, or $[11,1\bar{1}]$ if $g$ is Lorentzian.
%\end{theorem}
And for the Lorentzian signature we have:

\begin{theorem}\cite{Ca1}\label{ND-LO}
Let $(M,g)$ be a conformally flat homogeneous Lorentzian four-manifold with the Ricci operator of Segre type $[11,1\bar1]$. Then, $(M,g)$ is locally isometric to one of the unsolvable Lie groups $SU(2) \times \mathbb R$ or $SL(2,\mathbb R) \times \mathbb R$, equipped with a left invariant Lorentzian metric, admitting a pseudo-orthonormal basis  $\{e_1,e_2,e_3,e_4\}$ for the Lie algebra, such that the Lie brackets take one of the following forms:
$$
\begin{array}{llll}
{\rm i)} \quad & [e_1,e_2]=-2\alpha(\varepsilon e_3+e_4),\quad&[e_1,e_3]=\varepsilon\alpha e_2,\quad& [e_1,e_4]=\alpha e_2,\\[2pt] &
[e_2,e_3]=\varepsilon\alpha e_1,\quad & [e_2,e_4]=\alpha e_1,
\end{array}
$$$$
\begin{array}{llll}
{\rm ii)} \quad & [e_1,e_2]=2\alpha(\varepsilon e_3+e_4),\quad&[e_1,e_3]=\varepsilon\alpha e_2,\quad& [e_1,e_4]=\alpha e_2,\\[2pt]
& [e_2,e_3]=\varepsilon\alpha e_1,\quad & [e_2,e_4]=\alpha e_1,
\end{array}
$$
where $\alpha\neq 0$ is a real constant and $\varepsilon=\pm1$.
\end{theorem}

By using the above classification theorems we have enough tools to study Walker structures. The result is the following theorem.
\begin{theorem}\label{nondegWalker}
Let $(M,g)$ be a conformally flat homogeneous four-dimensional manifold with non-degenerate Ricci operator. Then $(M,g)$ does not admit any left-invariant Walker structure.
\begin{proof}
Since the Ricci operator of $(M,g)$ in non-degenerate, according to the Theorem \ref{ND-NU} for signature $(2,2)$ and Theorem \ref{ND-LO} for Lorentzian signature, we have the explicit description of Lie groups and their Lie algebras. We prove that the existence of a left invariant
parallel null distribution in any possible case leads to a contradiction. We report the
calculations for the case $(i)$ of signature $(2,2)$. Choose the pseudo-orthonormal basis $\{e_1,e_2,e_3,e_4\}$ and suppose there exists a two-dimensional
parallel distribution $\bar\D ={\rm span}( v,w)$ , where $v=\sum_{i=1}^4v_ie_i$ and $w=\sum_{i=1}^4w_ie_i$ are linearly independent and $g(v,v) =
g(w,w) = g(w,v) = 0$ for arbitrary parameters $v_i$ and $w_i$. Setting $\Lambda_i=\nabla_{e_i}$, the components of the Levi-Civita connection are calculated using the well known {\em Koszul} formula and are

 $$
 \Lambda_1= \left( \begin{array}{cccc}
   0 &0 & 0 & 0  \\
    0 & 0 &\alpha & 0  \\
   0 & \alpha & 0 & 0  \\
   0 & 0 & 0 & 0
 \end{array}  \right),\quad
  \Lambda_2= \left( \begin{array}{cccc}
   0 & 0 & \alpha(1-\varepsilon) & 0  \\
   0 & 0 & 0 & 0  \\
   \alpha(1-\varepsilon) & 0 & 0 & -\alpha(1 +\varepsilon)  \\
   0 & 0 & \alpha(1+\varepsilon) & 0
 \end{array}  \right),
$$
$$
 \Lambda_3= \left( \begin{array}{cccc}
   0 & 0 & -\alpha(1+\varepsilon) & 0  \\
   \alpha(1+\varepsilon) & 0 & 0 & \alpha(1-\varepsilon)  \\
   0 &  0 & 0 & 0  \\
   0 & 0 &  \alpha(1-\varepsilon) & 0
 \end{array}  \right),\quad
  \Lambda_4= \left( \begin{array}{cccc}
 0 &0 & 0 & 0  \\
    0 & 0 &-\varepsilon\alpha & 0  \\
   0 & -\varepsilon\alpha & 0 & 0  \\
   0 & 0 & 0 & 0
 \end{array}  \right).\\
 $$
Being parallel of $\D$ is expressed by the equations
\begin{equation}\label{ParallelDis}
\begin{array}{llll}
 \quad &\nabla_{e_{1}}v=a_{1}v+b_{1}w,\quad &\nabla_{e_{1}}w=c_{1}v+d_{1}w,\\
\quad &\nabla_{e_{2}}v=a_{2}v+b_{2}w,\quad &\nabla_{e_{2}}w=c_{2}v+d_{2}w,\\
\quad &\nabla_{e_{3}}v=a_{3}v+b_{3}w,\quad &\nabla_{e_{3}}w=c_{3}v+d_{3}w,\\
\quad &\nabla_{e_{4}}v=a_{4}v+b_{4}w,\quad &\nabla_{e_{4}}w=c_{4}v+d_{4}w,\\
\end{array}
\end{equation}
for some parameters $\{a_i,b_i,c_i,d_i\}_{i=1}^4$. From $g(v,v) =
g(w,w) = g(w,v) = 0$ and the equations $\nabla_{e_{1}}v=a_{1}v+b_{1}w$ and $\nabla_{e_{2}}v=a_{2}v+b_{2}w$ we have:
$$
\begin{array}{l}
v_1^2+v_2^2-v_3^2-v_4^2=0,\quad w_1^2+w_2^2-w_3^2-w_4^2=0,\quad v_1w_1+v_2w_2-v_3w_3-w_4v_4=0,\\
b_1w_1+a_1v_1=0,\quad b_1w_4+a_1v_4=0,\quad b_1w_2+a_1v_2-\alpha v_3=0,\\
b_1w_3+a_1v_3-\alpha v_2=0,\quad b_2w_2+a_2v_2=0,\quad b_2w_1+a_2v_1-\alpha v_3(1-\varepsilon)=0,\\
b_2w_4+a_2v_4-\alpha v_3(1+\varepsilon)=0,\quad b_2w_3+a_2v_3+\alpha(v_1+v_4)(1-\varepsilon)=0.
\end{array}
$$
These equations yield that the vector $v$ must vanish which contradicts the linear independence of $v,w$. By similar argument we suppose that $\D={\rm span}(x)$ is a null parallel line field, where $x=\sum_{i=1}^4x_ie_i$ for arbitrary parameters $x_i$. Thus, the following equations must be satisfied for some parameters $\omega_i$ and $x_i$,
$$
\begin{array}{l}
x_1^2+x_2^2-x_3^2-x_4^2=0,\\
\omega_1x_1=0,\quad \omega_1x_4=0,\quad \omega_1x_2-\alpha x_3=0,\quad \omega_1x_3-\alpha x_2=0,\\
\omega_2x_2=0,\quad \omega_2x_1+\alpha x_3(\varepsilon-1)=0,\quad \omega_2x_4-\alpha x_3(\varepsilon+1)=0,\\
\omega_2x_3+\alpha x_4(\varepsilon+1)+\alpha x_1(\varepsilon-1)=0,\\
\omega_3x_3=0,\quad \omega_3x_1+\alpha x_2(\varepsilon+1)=0,\quad \omega_3x_4+\alpha x_2(\varepsilon-1)=0,\\
\omega_3x_2+\alpha x_4(\varepsilon-1)-\alpha x_1(\varepsilon+1)=0,\\
\omega_4x_1=0,\quad \omega_4x_4=0,\quad \omega_4x_2+\alpha x_3\varepsilon=0,\quad \omega_4x_3+\alpha x_2\varepsilon=0.
\end{array}
$$
This system of equations yields that $x=0$ which is a contradiction. This shows that no left-invariant parallel null line field exists in this case and this matter finishes the proof.
\end{proof}
\end{theorem}

\section{Cases with degenerate Ricci operator and trivial isotropy}
Following the arguments of the previous section, now we proceed the manifolds with degenerate Ricci operator. First the Ricci parallel examples. By Proposition 4.1 of \cite{Ca1}, a conformally flat Ricci parallel homogeneous pseudo-Riemannian four-manifold $(M,g)$ belongs to one of the following cases:
\begin{itemize}
\item[(i)] if the Ricci operator is diagonalizable then $(M,g)$ is locally isometric to one of the four-dimensional spaces listed in the Theorem \ref{CFdiag}.
\item[(ii)] if the Ricci operator is not diagonalizable then
\begin{itemize}
\item[a)] either $(M,g)$ is locally isometric to a {\em complex sphere} in $\mathbb C^{3}$, defined by
$$z_1 ^2+z_2 ^2+z_3 ^2=i b \qquad (b\neq 0, \; b \in \mathbb R ),$$
\item[b)] or $(M,g)$ is a (conformally flat, locally symmetric) Walker manifold. In this case, $Q$ is two-step nilpotent, that is, $Q^2=0$.
\end{itemize}
\end{itemize}
Thus, the conformally flat Ricci parallel homogeneous Walker spaces are one of the spaces of the Theorem \ref{CFdiag}, or admit a two step nilpotent Ricci operator.

Now, let $(M,g)$ be a not Ricci parallel (and so not locally symmetric) conformally flat homogeneous manifold with degenerate Ricci operator. First, we proceed that the cases of trivial isotropy. Except the diagonalizable Ricci operator and Ricci parallel cases, spaces with trivial isotropy are locally isometric to a Lie group $G$, equipped with a left-invariant neutral metric, and $Q$ has one of the Segre types: $[1,(12)]$, $[(1,12)]$, $[(22)]$, $[(13)]$ and $[(1,3)]$. Also, for the Lorentzian signature, $Q$ admits the Segre types either $[(11,2)]$, or $[(1,3)]$ (see \cite{Ca1}).

\begin{theorem}\label{DegTriv1}
Let $(M,g)$ be a conformally flat not Ricci-parallel four-dimensional Lie group with the Ricci operator of Segre type $[1,(12)]$ or $[(1,3)]$, then $(M,g)$ does not admit a left-invariant Walker structure.
\end{theorem}
\begin{proof}
We apply the same argument used to prove Theorem \ref{nondegWalker}. As a sample, we prove that in the case with Segre type $[(1,3)]$ and signature $(2,2)$ there exists no two-dimensional null parallel distribution. To note that the statement of the theorem is valid for the Segre types $[(1,3)]$ in both Lorentzian and neutral signatures.

Using the notation of Theorem 4.3 of \cite{Ca1}, $(M,g)$ is isometric to the solvable Lie group $\mathbb{R}\ltimes E(1,1)$, where the Lie algebra $\g$ is
$$\begin{array}{ll}
[e_1,e_2] =(c_1-c_2)e_2-\frac{\sqrt2}{4c_2}e_3+(c_1-c_2)e_4, &   [e_1,e_3]=\frac{3\sqrt2}{4c_2}e_2-c_2e_3+\frac{3\sqrt2}{4c_2}e_4,\\[6 pt]
[e_1,e_4]=-(c_1+c_2)e_2+\frac{\sqrt2}{4c_2}e_3-(c_1+c_2)e_4, & [e_2,e_4]=\sqrt{2}\phi c_2(e_2+e_4), \\[6 pt]
[e_2,e_3]=[e_3,e_4]=-\frac{3\phi}{4c_2}e_2+\frac{\sqrt{2}\phi c_2}{2}e_3-\frac{3\phi}{4c_2}e_4,
\end{array}$$
where $\phi=\pm \frac{\sqrt{1-2c_1c_2^3}}{c_2^2}$, for any real constants $c_1,c_2\neq0$ such that $1-2c_1c_2^3\geq0$ and $\{e_1,e_2,e_3,e_4\}$ with $e_3,e_4$ time-like, is a pseudo-orthonormal basis. The components of the Levi-Civita connection are
 $$
 \Lambda_1= \left( \begin{array}{cccc}
   0 &0 & 0 & 0  \\
    0 & 0 & \frac{\sqrt{2}}{4c_2} & -c_2  \\
   0 & \frac{\sqrt{2}}{4c_2} & 0 & -\frac{\sqrt{2}}{4c_2}  \\
   0 & -c_2 & \frac{\sqrt{2}}{4c_2} & 0
 \end{array}  \right),
 $$
 $$
  \Lambda_2= \left( \begin{array}{cccc}
   0 & c_1-c_2 & \frac{\sqrt{2}}{2c_2} & -c_1  \\
    -c_1+c_2 & 0 & \frac{-3 \sqrt{1-2c_1 c_2^3}}{4c_2^3} & \frac{\sqrt{2} \sqrt{1-2c_1 c_2^3}}{c_2}  \\
   \frac{\sqrt{2}}{2c_2} & \frac{-3 \sqrt{1-2c_1 c_2^3}}{4c_2^3} & 0 & \frac{3 \sqrt{1-2c_1 c_2^3}}{4c_2^3}  \\
   -c_1 & \frac{\sqrt{2} \sqrt{1-2c_1 c_2^3}}{c_2} & \frac{-3 \sqrt{1-2c_1 c_2^3}}{4c_2^3} & 0
 \end{array}  \right),
$$
$$
 \Lambda_3= \left( \begin{array}{cccc}
    0 & \frac{\sqrt{2}}{2c_2} & c_2 & \frac{-\sqrt{2}}{2c_2}  \\
    \frac{-\sqrt{2}}{2c_2} & 0 &  \frac{-\sqrt{2} \sqrt{1-2c_1 c_2^3}}{2c_2} & 0  \\
   c_2 &  \frac{-\sqrt{2} \sqrt{1-2c_1 c_2^3}}{2c_2} & 0 &  \frac{\sqrt{2} \sqrt{1-2c_1 c_2^3}}{2c_2}  \\
    \frac{-\sqrt{2}}{2c_2} & 0 &  \frac{-\sqrt{2} \sqrt{1-2c_1 c_2^3}}{2c_2} & 0
 \end{array}  \right),
 $$
 $$
  \Lambda_4= \left( \begin{array}{cccc}
  0 &- c_1 & \frac{-\sqrt{2}}{2c_2} & c_1+c_2  \\
    c_1 & 0 & \frac{3 \sqrt{1-2c_1 c_2^3}}{4c_2^3} & \frac{-\sqrt{2} \sqrt{1-2c_1 c_2^3}}{c_2}  \\
   \frac{-\sqrt{2}}{2c_2} & \frac{3 \sqrt{1-2c_1 c_2^3}}{4c_2^3} & 0 & \frac{-3 \sqrt{1-2c_1 c_2^3}}{4c_2^3}  \\
   c_1+c_2 & \frac{-\sqrt{2} \sqrt{1-2c_1 c_2^3}}{c_2} & \frac{3 \sqrt{1-2c_1 c_2^3}}{4c_2^3} & 0
 \end{array}  \right).\\
 $$
Suppose that the null parallel distribution is $\bar\D={\rm span}(v,w)$. Since $v$ is parallel along the vectors $e_1$ and $e_2$ together with being null of $\bar\D$ gives the only possibility is to vanish $v$. This contradicts the linear independence of $v$ and $w$ and so, no left-invariant null parallel distribution also exists in this case. Similar argument concludes that no left-invariant null parallel line field exists and so $(M,g)$ is not a Walker manifold.
\end{proof}

Following \cite{Ca1}, for the other possible Segre types of the Ricci operator and the explicit solution presented there, we have the following result.

\begin{theorem}
Let $(M,g)$ be a conformally flat not Ricci-parallel four-dimensional Lie group with the Ricci operator of Segre type $[(1,12)]$, $[(22)]$ or $[(11,2)]$. Then $(M,g)$ may be a Walker manifold. An explicit example is
\begin{itemize}
\item[1-] {\bf Segre type $[(1,12)]$} the solvable Lie group $G=\mathbb R \ltimes \mathbb R^3$, whose Lie algebra $\g$ is described by
$$\begin{array}{ll}
[e_1,e_2]=- [e_1,e_3]=-\frac{1}{2c_1}e_1-c_2e_2-c_2e_3, \quad &  [e_2,e_3]= \frac{2c_1^2+1}{2c_1}e_2+\frac{2c_1^2+1}{2c_1}e_3, \\[4 pt]
[e_2,e_4]=-[e_3,e_4]=c_3e_2+c_3e_3+c_1e_4,
\end{array}$$
for any real constants $c_1 \neq 0, c_2,c_3$. In this case a left-invariant parallel degenerate line filed is given by $\D={\rm span}( e_{2}+e_{3})$ and $\bar\D={\rm span}( e_{2}+e_{3},e_{1}-e_{4})$ generates a left-invariant parallel degenerate plane field.
\item[2-] {\bf Segre type $[(22)]$} the solvable Lie group $G=\mathbb R \ltimes E(1,1)$, whose Lie algebra $\g$ is described by
$$\begin{array}{ll}
[e_1,e_2]=\frac{1}{2}e_4,\; & [e_1,e_3]=\frac{3}{2}(e_1+e_3),\\[4 pt]
[e_1,e_4]= \frac{5}{4}e_2+e_4,\; & [e_2,e_3]=e_2+\frac{5}{4}e_4,\\[4 pt]
[e_3,e_4]=-\frac{1}{2}e_2,
\end{array}$$
for any real constant $c_1 \neq 0$. In this case $\bar\D={\rm span}( e_{1}+e_{3},e_{2}+e_{4})$ generates a left-invariant parallel degenerate plane field.
\item[3-] {\bf Segre type $[(11,2)]$} the solvable Lie group $G=\mathbb R \ltimes H$, where $H$ is the Heisenberg group, whose Lie algebra $\g$ is described by
$$
\begin{array}{ll}
[e_1,e_2]=c_1e_3+c_1e_4,\quad & [e_1,e_3]=-[e_1,e_4]=-\frac{1}{2c_2}e_1-c_1e_2-c_3e_3-c_3e_4,\\[4 pt]
[e_3,e_4]={\frac{2c_2^2+1}{2c_2}(e_3+e_4)},\quad & [e_2, e_3]=-[e_2,e_4]=-c_2e_2+c_4e_3+c_4e_4,
\end{array}$$
for any real constants $c_1,c_3,c_4$ and $c_2 \neq 0$. In this case a left-invariant parallel degenerate line field is generated by $\D={\rm span}( e_{3}+e_{4})$.

\end{itemize}
\end{theorem}
\begin{proof}
An explicit example of a non-Ricci parallel conformally flat Lie group for the Segre types mentioned in the statement presented in \cite{Ca1}. We bring the details of the case of Segre type $[(1,12)]$. The other examples could be checked by similar calculations. The components of the Levi-Civita connection are calculated using the {\em Koszul} formual, which are
 $$
\Lambda_1= \left( \begin{array}{cccc}
   0 &\frac{-1}{2c_1} & \frac{1}{2c_1} & 0  \\
    \frac{1}{2c_1} & 0 &0 & 0  \\
  \frac{1}{2c_1} &  & 0 & 0  \\
   0 & 0 & 0 & 0
 \end{array}  \right),
  \Lambda_2=-\Lambda_3= \left( \begin{array}{cccc}
  0 &-c_2 & c_2 & 0  \\
    c_2 & 0 &\frac{1+2c_1^2}{2c_1} & c_3  \\
   c_2 & \frac{1+2c_1^2}{2c_1} & 0 & c_3  \\
   0 & c_3 & -c_3 & 0
 \end{array}  \right),
$$
$$
  \Lambda_4= \left( \begin{array}{cccc}
 0 &0 & 0 & 0  \\
    0 & 0 &0 & c_1  \\
   0 & 0 & 0 & -c_1  \\
   0 & -c_1 & c_1 & 0
 \end{array}  \right).\\
 $$
If we set $v= e_{1}-e_{4}$ and $w=e_{2}+e_{3}$ then have $g(v,v) =
g(w,w) = g(w,v) = 0$ which shows that $\bar\D$ is a null distribution. Also, direct calculations yields that
$$
\begin{array}{llll}
 \quad &\nabla_{e_{1}}v=\frac{1}{2c_{1}}w,\quad &\quad &\nabla_{e_{1}}w=0,\\
\quad &\nabla_{e_{2}}v=(c_2-c_3)w,\quad &\quad &\nabla_{e_{2}}w=\frac{1+2c_{1}^{2}}{2c_{1}}w,\\
\quad &\nabla_{e_{3}}v=(c_3-c_2)w,\quad &\quad &\nabla_{e_{3}}w=-\frac{1+2c_{1}^{2}}{2c_{1}}w,\\
\quad &\nabla_{e_{4}}v=c_1w,\quad &\quad &\nabla_{e_{4}}w=0,\\
\end{array}
$$
so, $\bar\D = {\rm span}( v,w)$ is a two-dimensional parallel null distribution. It is clear from the above derivations that $w$ generates a null parallel line field and this finishes the proof.
\end{proof}

\begin{rem}
According to standard calculations similar to the Theorem \ref{DegTriv1}, for the explicit example which is presented in \cite{Ca1} for the segre type $[(13)]$, there is not any left-invariant Walker structure. This fact does not mean that conformally flat homogeneous invariant Walker manifolds can not admit this Segre type but the example in hand is not an invariant Walker manifold.
\end{rem}
%%
%
%In the cases with neutral signature for the Ricci operator of Segre types $[1,(12)]$, $[(13)]$ and $[(1,3)]$ there is no any Walker structure. Also in the cases with Lorentzian signature for the Ricci operator of Segre type $[(1,3)]$ there is no any Walker structure. But for other case we have the following result.
%\begin{theorem}
%A conformally flat, not Ricci-parallel, homogeneous four manifold $(M,g)$, with degenerate and not diagonalizale Ricci operator admits a parallel degenerate line field $D$ if and only if one of the following cases occurs:
%
%
%$(i)$ $(M,g)$ is neutral and of Segre type $[1,12]$. In this case, $D={\rm span}\lbrace u_{2}+u_{3}\rbrace$.
%
%
%$(ii)$ $(M,g)$ is Lorentzian and of Segre type $[11,2]$. In this case, $D={\rm span}\lbrace u_{3}+u_{4}\rbrace$.
%
%\end{theorem}
%\begin{theorem}
%A conformally flat, not Ricci-parallel, homogeneous neutral four manifold $(M,g)$, with degenerate and not diagonalizale Ricci operator admits a parallel null plane field $D$ if and only if one of the following cases occurs:\\
%
%$(i)$ $(M,g)$ is of Segre type $[1,12]$. In this case, $D={\rm span}\lbrace u_{2}+u_{3},\frac{1}{c_{2}-c_{3}}(u_{1}-u_{4})\rbrace$.\\
%$(ii)$ $(M,g)$ is of Segre type $[22]$ and $c_{1}=\frac{1}{2}$. In this case, $D={\rm span}\lbrace u_{1}+u_{3},u_{2}+u_{4}\rbrace$.\\
%
%\end{theorem}
\section{Cases with degenerate Ricci operator and non-trivial isotropy}
After study the spaces with degenerate Ricci operator and trivial isotropy we consider cases with non-trivial isotropy in this section. For these spaces, the approach is based on the classification of four dimensional homogeneous spaces with non-trivial isotropy presented by Komrakov in \cite{Ko}. In \cite{Ca1}, the authors checked case by case the Komrakov's list for conformally flat not Ricci parallel (and so not locally symmetric) examples with degenerate and not diagonalizable Ricci operator. The following theorem shows the possible Segre types.
\begin{theorem}
\cite{Ca1} Let $(M,g)$ be a conformally flat homogeneous, not locally symmetric pseudo-Riemannian four manifold, whose  Ricci operator $Q$ is degenerate and not diagonalizale. Then $Q$ is of Segre type either $[(22)]$,$[(1,12)]$ or $[(11,2)]$.
\end{theorem}
By using the lists which are presented in \cite{Ca1} for the conformally flat non-symmetric homogeneous 4-spaces with non-trivial isotropy and non-diagonalizable degenerate Ricci operator, now we are able to determine the invariant Walker structures over these spaces.
\begin{theorem}
 Let $(M,g)$ be a conformally flat homogeneous not locally symmetric pseudo-Riemannian four-manifold with not diagonalizable, degenerate Ricci operator and non-trivial isotropy. Then $(M,g)$ admits invariant parallel degenerate line field $\D$ and invariant parallel null plane field $\bar\D$ with the generators listed in the Tables {\rm III}, {\rm IV} and {\rm V}.
\end{theorem}
\begin{proof} Following the notation and the classifcation used in \cite{Ko}, the space identified by the
type $n.m^k:q$ is the one corresponding to the $q$-th pair $(\g,\h)$ of type $n.m^k$, where $n={\rm dim}(\h)$
$(= 1,..., 6)$, $m$ is the number of the complex subalgebra $\h^{\mathbb C}$ of ${\mathfrak so}(4,\mathbb C)$ and $k$ is the number
of the real form of $\h^{\mathbb C}$. According to the lists which are specified in \cite{Ca1} for the conformally flat homogeneous not locally symmetric four-manifolds with non-trivial isotropy and non-diagonalizable Ricci operator we check case by case the Walker structures and prepare the list below for the Walker examples.
In each of the different cases, $\lbrace u_{1},u_{2},u_{3},u_{4}\rbrace$ is the basis of $\m$ used in \cite{Ko} in the description
of the quotient space $M=\frac{G}{H}$, $\lbrace \omega_{1},\omega_{2},\omega_{3},\omega_{4}\rbrace$ the corresponding dual basis of one-
forms. Moreover, $\omega_{i}\omega_{j}$ denote the symmetric tensor product of $\omega_{i}$
and $\omega_{j}$. We bring the details of the case $1.3^1$:$2$ here and prove that in this case, $\bar\D ={\rm span}( u_1,u_2) $ is a two-dimensional
invariant null parallel distribution. By the table {\rm III} of \cite{Ca1}, in this case
the Lie algebra $\g$ is described by the brackets as follows
$$
[u_1,u_3]=-\lambda e_1+(\lambda+1)u_1+\lambda u_2,\quad [u_2,u_4]=u_2,\quad [e_1,u_3]=u_1,\quad [e_1,u_4]=u_2,
$$
where $\h={\rm span}(e_1)$ and $\lambda\neq0$ is an arbitrary parameter. The isotropy representation $H$ and the pseudo-Riemannian invariant metric $g$, are described as
$$
\begin{array}{cc}
H=\left(\begin{array}{cccc}
0&0&1&0\\0&0&0&1\\0&0&0&0\\0&0&0&0
\end{array}\right),\quad&
(g_{ij})  = \left( \begin{array}{cccc}
   0 & 0 & 0 & -a  \\
   0 & 0 & a & 0  \\
   0 & a & b & c  \\
   -a & 0 & c & 0
 \end{array}  \right),
\end{array}
$$
for arbitrary parameters $a,b,c$. Direct calculations using the {\em Koszul} formula yield the following Levi-Civita connection
$$
\begin{array}{ll}
\Lambda_1=\left( \begin {array}{cccc}
0&0&\frac{\lambda+1}{2}&0\\
0&0&0&\frac{\lambda+1}{2}\\
0&0&0&0\\
0&0&0&0\end {array} \right),&
\Lambda_2=\left( \begin {array}{cccc}
0&0&\frac{1}{2}&0\\
0&0&0&\frac{1}{2}\\
0&0&0&0\\
0&0&0&0\end {array} \right),\\
\Lambda_3=\left( \begin {array}{cccc}
-\frac{\lambda+1}{2}&\frac{1}{2}&0&\frac{c}{2a}\\
-\lambda&0&\frac{c\lambda}{a}&-\frac {
c\lambda+c+b}{2a}\\
0&0&0&\frac{1}{2}\\
0&0&-\lambda&\frac{\lambda+1}{2}\end {array} \right),&
\Lambda_4=\left(\begin {array}{cccc}
0&0&\frac{c}{2a}&0\\
\frac{\lambda+1}{2}&-\frac{1}{2}&-\frac{c\lambda+c+b}
{2a}&0\\
0&0&\frac{1}{2}&0\\
0&0&\frac{\lambda+1}{2}&0\end {array}\right).
\end{array}
$$
If we set $v=u_1$ and $w=u_2$, the non-zero covariant derivatives are
$$
\begin{array}{llll}
\nabla_{u_3}v=-\frac{\lambda+1}{2}v-\lambda w, &\nabla_{u_4}v=\frac{\lambda+1}{2}w,& \nabla_{u_3}w=\frac12 v,& \nabla_{u_4}w=-\frac12 w.
\end{array}
$$
Also, $Hv=Hw=0$, so, $\bar\D={\rm span}(v,w)$ is an invariant null parallel distribution since $g(v,v)=g(w,w)=g(v,w)=0$. On the other hand, set $x=\sum_{i=1}^4x_ie_i$ and suppose that $\D={\rm span}(x)$ is an invariant null parallel line field. Then, the following equations must satisfy for some parameters $\omega_1,\dots,\omega_4$
$$
\begin{array}{llll}
\nabla_{u_1}x=\omega_1x,&\nabla_{u_2}x=\omega_2x,&\nabla_{u_3}x=\omega_3x,&\nabla_{u_4}x=\omega_4x.
\end{array}
$$
By straight forward calculations we conclude that the following equations must satisfy
$$
\begin{array}{l}
\omega_1x_3=0,\quad \omega_1x_4=0,\quad -\omega_1x_1+\frac{\lambda+1}{2}x_3=0,\quad -\omega_1x_2+\frac{\lambda+1}{2}x_4=0\\
\omega_2x_3=0,\quad \omega_2x_4=0,\quad -\omega_2x_1+\frac12x_3=0,\quad -\omega_2x_2+\frac12x_4=0,\\
2a\omega_3x_1-cx_4-ax_2+ax_1(\lambda+1)=0,\quad 2a\omega_3x_2+x_4(b+c(\lambda+1))-2c\lambda x_3+2a\lambda x_1=0,\\
\omega_3x_3-\frac12x_4=0,\quad x_4(\omega_3-\frac12(\lambda+1))+x_3\lambda=0,\quad 2a\omega_4x_1-cx_3=0,\quad x_3(\frac12-\omega_4)=0,\\
ax_2(2\omega_4+1)+bx_3-(\lambda+1)(ax_1-cx_3)=0,\quad -\omega_4x_4+\frac12x_3(\lambda+1)=0.
\end{array}
$$
By solving the above system of equations we obtain that $x$ must vanish which is a contradiction. Thus, no invariant parallel null line field exist in this case.
\end{proof}

 The following tables show the existence of invariant Walker structures on the non-Ricci parallel and non-diagonalizable Ricci operator conformally flat homogeneous spaces with non-trivial isotropy according to different Segre types of the Ricci operator. In these tables the column $\bar\D$ (respectively $\D$) shows the generators of the invariant null parallel plane field (respectively invariant null parallel line field) in each case and the sign {\underline\xmark} shows that the invariant Walker structure does not exist.

{\small
\medskip
{\bf Table I: Non-symmetric examples with ${\bf Q}$ of Segre type ${\bf [(22)]}$.}\nopagebreak \\[3 pt]
\begin{tabular}{!{\vrule width 1pt}p{1.1cm}|p{7cm}|p{1.5cm}|p{4cm}!{\vrule width 1pt}}
\noalign{\hrule height 1pt}
Case&Invariant metric&$\vphantom{A^{A^{A^A}}}\bar\D$&$\D$\\
\hline
$1.3^1$:$2$&$-2a\om_1\om_4+2a\om_2\om_3+b\om_3\om_3+2c\om_3\om_4$&$\lbrace u_{1},u_{2}\rbrace$&\xmark\\
\hline
$1.3^1$:$4$&$2a(-\om_1\om_4+\om_2\om_3)+b\om_3\om_3+2c\om_3\om_4$&$\lbrace u_{1},u_{2}\rbrace$&\xmark\\
\hline
$1.3^1$:$5$&$2a(-\om_1\om_4+\om_2\om_3)+\frac{2c\lambda\mu-d\lambda^2-\mu d-2c\lambda}{\mu(\mu-1)}\om_3\om_3\newline+2c\om_3\om_4+d\om_4\om_4$&$\{u_1,u_2\}$&\xmark\\
\hline
$1.3^1$:$7$&$2a(-\om_1\om_4+\om_2\om_3)+b\om_3\om_3+2c\om_3\om_4\newline+
(b\lambda-2c)\om_4\om_4$&$\lbrace u_{1},u_{2}\rbrace$&\xmark\\
\hline
$1.3^1$:$15$&$2a(-\om_1\om_4+\om_2\om_3)-d\om_3\om_3
+2c\om_3\om_4+d\om_4\om_4$&$\lbrace u_{1},u_{2}\rbrace$&\xmark\\
\hline
$1.3^1$:$16$&$2a(-\om_1\om_4+\om_2\om_3)+
d\om_3\om_3+2c\om_3\om_4+d\om_4\om_4$&$\lbrace u_{1},u_{2}\rbrace$&\xmark\\
\hline
$1.3^1$:$24$&$2a(-\om_1\om_4+\om_2\om_3)+2d(\lambda^2
-\lambda)\om_3\om_3+2c\om_3\om_4\newline+d\om_4\om_4$&$\{u_1,u_2\}$&\xmark\\
\hline
$1.3^1$:$25$&$2a(-\om_1\om_4+\om_2\om_3)-2d(\lambda^2-\lambda)\om_3\om_3+
2c\om_3\om_4\newline+d\om_4\om_4$&$\{u_1,u_2\}$&\xmark\\
\hline
$1.3^1$:$28$&$2a(-\om_1\om_4+\om_2\om_3)+2d\om_3\om_3+
2c\om_3\om_4+d\om_4\om_4$&$\lbrace u_{1},u_{2}\rbrace$&\xmark\\
\hline
$1.3^1$:$29$&$2a(-\om_1\om_4+\om_2\om_3)-2d\om_3\om_3+
2c\om_3\om_4+d\om_4\om_4$&$\lbrace u_{1},u_{2}\rbrace$&\xmark\\
\hline
$1.3^1$:$30$&$2a(-\om_1\om_4+\om_2\om_3)+b(\lambda^2-
\lambda)\om_3\om_3\newline-(b\mu+d\lambda-d-b)\om_3\om_4
+d\om_4\om_4$&$\{u_1,u_2\}$&$\begin{array}{c}(\mu-1)u_1+u_2\rm{\ for\ }\lambda=0\\ u_1-u_2{\rm\ for\ }\lambda=-\mu\\u_1+(\lambda-1)u_2{\rm\ for\ }\mu=0\end{array}$\\
\noalign{\hrule height 1pt}
\end{tabular}
}

{\small\medskip
{\bf Table II: Non-symmetric examples with $\bf{Q}$ of Segre type $\bf{[(1,12)]}$.} \nopagebreak\\[3 pt]
\begin{tabular}{!{\vrule width 1pt}p{1.1cm}|p{7cm}|p{3cm}|p{2cm}!{\vrule width 1pt}}
\noalign{\hrule height 1pt}
Case&Invariant metric&$\vphantom{A^{A^{A^A}}}\bar\D$&$\D$\\
\hline
$1.1^1$:$1$ $\vphantom{\displaystyle\frac{A^a}{A^a}}$&$2a\om_1\om_3+2c\om_2\om_4+d\om_4\om_4$&\xmark&$\lbrace u_{2}\rbrace$\\ \hline
$1.1^1$:$2$ $\vphantom{\displaystyle\frac{A^a}{A^a}}$&$2a\om_1\om_3+2c\om_2\om_4+d\om_4\om_4$&\xmark&$\lbrace u_{2}\rbrace$\\ \hline
$1.3^1$:$5$ $\vphantom{\displaystyle\frac{A^a}{A^a}}$&$2a(-\om_1\om_4+\om_2\om_3)+b\om_3\om_3+2c\om_3\om_4-\frac{2c}{\lambda}\om_4\om_4$ &$\lbrace u_{1},u_{2}\rbrace$&$\{u_2\}$\\
\hline
$1.3^1$:$7$ $\vphantom{\displaystyle{A^A}}$ $\vphantom{\displaystyle\frac{A^a}{A^a}}$&$2a(-\om_1\om_4+\om_2\om_3)+b\om_3\om_3+2c\om_3\om_4-2c\om_4\om_4$ &$\lbrace u_{1},u_{2}\rbrace$&$\lbrace u_{2}\rbrace$\\
\hline
$1.3^1$:$12$&$2a(-\om_1\om_4+\om_2\om_3)+2c\om_3\om_4+d\om_4\om_4$&
\xmark&$\lbrace u_{1}\rbrace$ $\vphantom{\displaystyle\frac{A^a}{A^a}}$\\
\hline
$1.3^1$:$12$ $\vphantom{\displaystyle\frac{A^a}{A^a}}$&
$2a(-\om_1\om_4+\om_2\om_3)+b\om_3\om_3+2c\om_3\om_4+d\om_4\om_4$&
$\{u_1,u_2\}$&$\lbrace u_{1}\rbrace$\\
\hline
$1.3^1$:$12$ $\vphantom{\displaystyle\frac{A^a}{A^a}}$&
$2a(-\om_1\om_4+\om_2\om_3)+b\om_3  \om_3+2c\om_3\om_4+
d\om_4\om_4$&$\lbrace u_{1},u_{2}\rbrace$&$\lbrace u_{1}\rbrace$\\
\hline
$1.3^1$:$19$ $\vphantom{\displaystyle\frac{A^a}{A^a}}$&$2a(-\om_1\om_4+\om_2\om_3)+2c\om_3\om_4+d\om_4\om_4$&$
\lbrace u_{1},u_{2}\rbrace$&$\lbrace u_{1}\rbrace$\\
\hline
{$1.3^1$:$21$}$\vphantom{\displaystyle\frac{A^a}{A^a}}$&$2a(-\om_1\om_4+\om_2\om_3)+2c\om_3\om_4+d\om_4\om_4$&
$\lbrace u_{1},u_{2}\rbrace$&$\lbrace u_{1}\rbrace$\\
\hline
$1.3^1$:$21$ $\vphantom{\displaystyle\frac{A^a}{A^a}}$&
$2a(-\om_1\om_4+\om_2\om_3)+b\om_3\om_3+2c\om_3\om_4+d\om_4\om_4$&$\{u_1,u_2\}$&$\lbrace u_{1}\rbrace$\\
\hline
$1.3^1$:$24$ $\vphantom{\displaystyle\frac{A^a}{A^a}}$&
$2a(-\om_1\om_4+\om_2\om_3)+b\om_3\om_3+2c\om_3\om_4+d\om_4\om_4$&$
\lbrace u_{1},u_{2}\rbrace$&\xmark\\
\hline
$1.3^1$:$25$ $\vphantom{\displaystyle\frac{A^a}{A^a}}$&
$2a(-\om_1\om_4+\om_2\om_3)+b\om_3\om_3+2c\om_3\om_4+d\om_4\om_4$&$
\lbrace u_{1},u_{2}\rbrace$&\xmark\\
\hline
$1.3^1$:$30$ $\vphantom{\displaystyle\frac{A^a}{A^a}}$&
$2a(\om_2\om_3-\om_1\om_4)+b\om_3\om_3+{b(1-\mu)}\om_3\om_4+d\om_4\om_4$&$
\lbrace u_{1},u_{2}\rbrace$&$\lbrace u_{1}\rbrace$\\
\hline
$1.3^1$:$30$ $\vphantom{\displaystyle\frac{A^a}{A^a}}$&
$2a(\om_2\om_3-\om_1\om_4)+b\om_3\om_3+{d(1-\lambda)}\om_3\om_4+d\om_4\om_4$&$
\lbrace u_{1},u_{2}\rbrace$&$\lbrace u_{2}\rbrace$\\
\hline
%
%$1.4^1$:$2$$\vphantom{\displaystyle\frac{A^a}{A^a}}$\newline  &$a(\om_2\om_2-2\om_1\om_3)+b\om_3\om_3+2c\om_3\om_4+d\om_4\om_4 \newline ad<0, \; b\neq 0$&\xmark&\xmark\\
%\hline
%
$1.4^1$:$9$ $\vphantom{\displaystyle{A^{A^{A^A}}}}$ &$a(-2\om_1\om_3+\om_2\om_2)+b\om_3\om_3+2c\om_3\om_4 \newline {-\frac{a(4r+1)}{4}\om_4\om_4},\quad $&\xmark&$\lbrace u_{1}\rbrace$\\
\hline
$1.4^1$:$10$ &$a(-2\om_1\om_3+\om_2\om_2)+b\om_3\om_3+2c\om_3\om_4 \newline +d\om_4\om_4,\quad ad<0$&\xmark&$\lbrace u_{1}\rbrace$ \\
\hline
$2.2^1$:$2$ $\vphantom{\displaystyle\frac{A^a}{A^a}}$&${2a(\om_1\om_3+\om_2\om_4)+b\om_2\om_2}$&\xmark&$\lbrace u_{4}\rbrace$ \\
\hline
$2.2^1$:$3$ $\vphantom{\displaystyle\frac{A^a}{A^a}}$&${2a(\om_1\om_3+\om_2\om_4)+b\om_2\om_2}$&\xmark&$\lbrace u_{4}\rbrace$\\
\hline
$2.5^1$:$4$ $\vphantom{\displaystyle\frac{A^a}{A^a}}$&${2a(\om_1\om_3+\om_2\om_4)+b\om_3\om_3}$&\xmark&$\lbrace u_{1}\rbrace$\\
\hline
$3.3^1$:$1$ $\vphantom{\displaystyle\frac{A}{A^a}}$&${2a(\om_1\om_3+\om_2\om_4)+b\om_3\om_3}$&\xmark&$\lbrace u_{1}\rbrace$\\
\noalign{\hrule height 1pt}
\end{tabular}
}

{\small
\smallskip
{\bf Table III: Non-symmetric examples with ${\bf Q}$ of Segre type ${\bf [(11,2)]}$.}\nopagebreak\\[3 pt]
\begin{tabular}{!{\vrule width 1pt}p{1.1cm}|p{6.5cm}|p{3cm}|p{3cm}!{\vrule width 1pt}}
\noalign{\hrule height 1pt}
Case&Invariant metric&$\vphantom{A^{A^{A^A}}}\bar\D$&$\D$\\
\hline
$1.1^2$:$1$$\vphantom{\displaystyle\frac{A^a}{A^a}}$&$c(\om_1\om_1+\om_3\om_3)+2b\om_2\om_4+d\om_4\om_4$&\xmark&$\lbrace u_{2}\rbrace$\\
\hline
$1.1^2$:$2$$\vphantom{\displaystyle\frac{A^a}{A^a}}$&$c(\om_1\om_1+\om_3\om_3)+2b\om_2\om_4+d\om_4\om_4$&\xmark&$\lbrace u_{2}\rbrace$\\
\hline
%
%$1.4^1$:$2$$\vphantom{\displaystyle A^{A^{A^a}}}$&$a(-2\om_1\om_3+\om_2\om_2)+b\om_3\om_3+2c\om_3\om_4\newline+d\om_4\om_4,\quad ad>0, \; b\neq 0$&\xmark&\xmark\\
%\hline
%
$1.4^1$:$9$ $\vphantom{\displaystyle{A^{A^{A^A}}}}$&$a(-2\om_1\om_3+\om_2\om_2)+b\om_3\om_3+2c\om_3\om_4
\newline{-\frac{a(4r+1)}{4}\om_4\om_4}$&\xmark&$\lbrace u_{1}\rbrace$\\
\hline
$1.4^1$:$10$ $\vphantom{\displaystyle\frac{A^a}{A}}$ &$a(-2\om_1\om_3+\om_2\om_2)+b\om_3\om_3+2c\om_3\om_4\newline+d\om_4\om_4,\quad ad>0$&\xmark&$\lbrace u_{1}\rbrace$\\
\hline
$2.5^2$:$2$ $\vphantom{\displaystyle\frac{A^a}{A^a}}$&$2a\om_1\om_3+a(\om_2\om_2+\om_4\om_4)+b\om_3\om_3$&\xmark&$\lbrace u_{1}\rbrace$\\
\hline
$3.3^2$:$1$ $\vphantom{\displaystyle\frac{A^a}{A^a}}$&$2a\om_1\om_3+a(\om_2\om_2+\om_4\om_4)+b\om_3\om_3$&\xmark&$\lbrace u_{1}\rbrace$\\
\noalign{\hrule height 1pt}
\end{tabular}
}

\end{document}